\documentclass[11pt]{article}
\usepackage{caption}
\usepackage{subfigure}
\usepackage{indentfirst}
\usepackage{latexsym}
\usepackage{hyperref}
\usepackage{graphicx}
\usepackage{mathrsfs}
\usepackage{bookmark}
\usepackage{enumerate}
\usepackage{amsfonts,amsthm,amssymb,mathtools}
\usepackage{amsmath}
\usepackage{latexsym, euscript, epic, eepic, color}
\usepackage{amssymb}

\usepackage{enumerate}
\usepackage[all]{xy}
\usepackage{setspace}
\allowdisplaybreaks

\usepackage{epstopdf}
\numberwithin{equation}{section}
\newtheorem{theorem}{Theorem}[section]

\newtheorem{definition}[theorem]{Definition}

\newtheorem{lemma}[theorem]{Lemma}

 \allowdisplaybreaks

\begin{document}
\baselineskip=16pt

\title{Distance signless Laplacian spectral radius and perfect matching in graphs and bipartite graphs}

\author{\textbf{Chang Liu, Jianping Li\textsuperscript{\thanks{Corresponding author.}} }\\
	\small \emph{College of Liberal Arts and Sciences, National University of Defense Technology,}\\[-0.8ex]
	\small \emph{Changsha, Hunan, 410073, P. R. China}\\[-0.4ex]
		\small\tt clnudt19@126.com\\
		\small\tt lijianping65@nudt.edu.cn\\
	}

\date{\today}

\maketitle

\begin{abstract}
 The distance matrix $\mathcal{D}$ of a connected graph $G$ is the matrix indexed by the vertices of $G$ which entry $\mathcal{D}_{i,j}$ equals the distance between the vertices $v_i$ and $v_j$. The distance signless Laplacian matrix $\mathcal{Q}(G)$ of graph $G$ is defined as $\mathcal{Q}(G)=Diag(Tr)+\mathcal{D}(G)$, where $Diag(Tr)$ is the diagonal matrix of the vertex transmissions in $G$. The largest eigenvalue of $\mathcal{Q}(G)$ is called the distance signless Laplacian spectral radius of $G$, written as $\eta_1(G)$.  And a perfect matching in a graph is a set of disadjacent edges covering every vertex of $G$. In this paper, we present two suffcient conditions in terms of the distance signless Laplacian sepectral radius for the exsitence of perfect matchings in graphs and bipatite graphs. 
\\[2pt]
\end{abstract}



\section{Introduction}

In this paper we deal only with simple, connected, undirected graphs. Let $G$ be a graph with vertex set $V(G)$ and $E(G)$. The order of $G$ is the number $n=|V(G)|$ of its vertices. The distance matrix of $G$, denoted by $\mathcal{D}(G)$, is a $n\times n$ matrix indexed by the vertices of $G$ which entry $\mathcal{D}_{i,j}$ is the distance between the vertices $v_i$ and $v_j$. The transmissions $Tr(v)$ of a vertex $v\in V(G)$ is the sum of distances from $v$ to all vertices in $G$. Analogous to the approach used for the adjacency matrix by defining a Laplacian and a signless Laplacian, in \cite{Aou}, Aouchiche and Hansen gave definitions of distance Laplacian and distance signless Laplacian. The distance signless Laplacian matrix of a graph $G$ is given by $\mathcal{Q}(G)=Diag(Tr)+\mathcal{D}(G)$, where $Diag(Tr)$ is the diagonal matrix of the vertex transmissions in $G$. Let $\eta_1(G),\eta_2(G),\cdots,\eta_n(G)$ be eigenvalues of $\mathcal{Q}(G)$ in nondecreasing order. The largest eigenvalue $\eta_1(G)$ is also called the distance signless Laplacian spectral radius of graph $G$.

A matching in a graph $G$ is a subset of $E(G)$ such that any two edges in matchings has no commont vertices, and a perfect matching is a matching covering all vertices of $G$. For the existence of perfect matching in arbitrary graphs $G$, Tutte \cite{Tut} first provided a sufficient and necessary condition in terms of odd components.

\begin{theorem}[Tutte's 1-factor Theorem \cite{Tut}]\label{tutte}
	A graph $G$ contains a perfect matching if and only if for each subsets $S\subseteq V(G)$, $o(G-S)\leq|S|$, where $o(G-S)$ is the number of odd components in graph $G-S$.
\end{theorem}
  Moreover, for bipartite graphs $G=(X,Y)$, where $X$ and $Y$ are two vertex sets such that $X\cup Y =V(G)$, the essential and sufficient condition for the existence of a matching which cover one partite of $G$ was first given by Hall \cite{Hall}. 
  
 \begin{theorem}[Hall's theorem \cite{Hall}]\label{Hall}
 	A bipartite graph $G=(X,Y)$ has a matching which covers every vertex in $X$ if and only if for each subsets $S\subseteq X$, $|N(S)|\geq|S|$, where $N(S)$ is the set of all neighbours of the vertices in $S$.
 \end{theorem}

  Note that a connected bipartite graph $G=(X,Y)$ is balanced if $|X|=|Y|=n$. At this time, the matching mentioned in Theorem \ref{Hall} is a perfect matching. 
  
  Inspried by Tutte's 1-factor Theorem, researchers have tried a variety of ways to illustrate the existence of perfect matchings in graphs. Investigations on the relationship between eigenvalues and the perfect matching in graphs originated from the paper \cite{SO} published by O in 2020. He proved a lower bound for the adjancency spectral radius of $G$ to guarantee the existence of a perfect matching. Later, Liu et al. \cite{Liu} and Zhao et al. \cite{Zhao} successively obtained some sufficient conditions for the exsitence of perfect matchings in graph $G$, based on the signless Laplacian spectral radius and $A_{\alpha}$-spectral radius, respectively. 
  
  Vary recently, Zhang and Lin \cite{zhang} have gotten two upper bounds on distance spectral radius to ensure the existence of a perfect matching in graphs and balanced bipartite graphs, respectively. Along this line, we intend to generalize these interesting results by considering the distance signless Laplacian spectral radius. The main results in this paper are shown as below.

\begin{theorem}\label{them1}
	Let $G$ be a graph with $n$ vertex, where $n\geq4$ is an even number. The largest root of the equation  $x^3+(5-5n)x^2+(8n^2-25n+32)x-4n^3+22n^2-54n+48$ is denoted by $\theta(n)$. Then the graph $G$ contains a perfect matching if
 	\begin{equation*}
 		\eta_1(G)<\begin{cases}
 			\theta(n), &n=4 \mbox{ or } n \geq 12, \\
 			2n+\sqrt{\frac{n(n+2)}{2}}-2, & n=6,8,10. 
 		\end{cases}
 	\end{equation*} 
\end{theorem}

\begin{theorem}\label{them2}
	Let $G$ be a balanced bipartite graph with $2n$ vertices where $n\geq3$ is an integer, then $G$ contains a perfect matching if $\eta_1(G)<\kappa(n)$, where $\kappa(n)$ is the largest root of the equation $x^4+(12-18n)x^3+(119n^2-190n+76)x^2-(342n^3-915n^2+826n-252)x+360n^4-1383n^3+2026n^2-1362n+364=0$.
\end{theorem}
\section{Preliminaries}
In this section, we introduce some lemmas about the largest eigenvalue of real matrices, which will be used in our paper later. The largest eigenvalue of the matrix $M$ is denoted by $\rho(M)$.

\begin{lemma}\label{symm}
	If $M$ is an $n\times n$ real symmetric matrix, then $\rho(M)=\max\limits_{\textbf{X}\in \mathbb{R}^n}\dfrac{\textbf{X}^tM\textbf{X}}{\textbf{X}^{t}\textbf{X}}$.
\end{lemma}

\begin{lemma}\cite{Cgod}\label{largerrd}
	Let $M_1$ and $M_2$ be real nonnegative $n\times n$ matrices such that $M_1-M_2$ is nonnegative, then $\rho(M_2)\leq\rho(M_1)$. 
\end{lemma}

Next, we explain the concepts of equitable matrices and equitable partitions.
\begin{definition}
\label{def1}\cite{aebrouw}
	Let $M$ be a symmetric real matrix of order $n$ whose rows and columns are indexed by $P=\left\lbrace1,2,\cdots,n \right\rbrace $. $\left\lbrace P_1,P_2,\cdots,P_t \right\rbrace $ is a partition of $P$ with $n_i=|P_i|$ and $n=n_1+n_2+\cdots+n_t$. Let the matrix $M$ be partitioned according to $\left\lbrace P_1,P_2,\cdots,P_t \right\rbrace$, that is
\begin{equation*}
	M=\begin{pmatrix}
		M_{1,1} & M_{1,2} & \cdots & M_{1,t}\\
		M_{2,1} & M_{2,2} & \cdots & M_{2,t}\\
		\vdots & \vdots & \ddots & \vdots\\
		M_{t,1} & M_{t,2} & \cdots & M_{t,t}
	\end{pmatrix}_{n\times n},
\end{equation*}
	where the blocks $M_{i,j}$ denotes the submatrix of $M$ formed by rows in $P_i$ and the $P_j$ columns. Let $b_{i,j}$ denote the average row sum of $M_{i,j}$. Then the matrix $B=(b_{i,j})$ is called the quotient matrix of $M$ w.r.t. the given partition. Particularly, if the row sum of each submatrix $M_{i,j}$ is constant then the partition is called equitable.
\end{definition}

\begin{lemma}\cite{lhYou}\label{equit}
	Let $B$ be an equitable quotient matrix of $M$ as defined in Definition \ref{def1}. If $M$ is a nonnegative matrix, then $\rho(B)=\rho(M)$.
\end{lemma}

\section{Proof of Theorem \ref{them1}}
\begin{proof}[Proof of Theorem \ref{them1}]
	Assume to the contrary that $G$ contains no perfect matchings. By Theorem \ref{tutte}, one can suppose there exists a vertex set $S\in V(G)$ such that $o(G-S)-|S|>0$. Note that all components of $G-S$ are odd, otherwise we can delete one vertex in each even component and add these vertices to $S$, then the number of odd component and the order of $S$ have the same increase, so that $o(G-S)>|S|$ always holds and all components of $G-S$ are odd. For convenience, we denote $q=o(G-S)$ and $s=|S|$. Since $n$ is an even number, we have $q$ and $s$ have the same parity, and $q-s\geq2$.
	
	The joint of two graphs $G_1$ and $G_2$ is denoted by $G_1\vee G_2$, which is a graph such that $V(G_1\vee G_2)=V(G_1)\cup V(G_2)$ and $E(G_1\vee G_2)=E(G_1)\cup E(G_2)\cup \left\lbrace uv|~u\in V(G_1), v\in V(G_2) \right\rbrace $, and the disjoint of two graphs $G_1$ and $G_2$ is written as $G_1\cup G_2$, which is a graph such that $V(G_1\cup G_2)=V(G_1)\cup V(G_2)$ and $E(G_1\cup G_2)=E(G_1)\cup E(G_2)$. The complement of $G$ is denoted by $\overline{G}$. 
	
	Let $G_1,G_2,\cdots,G_q$ be the components of $G-S$ with $n_i=|G_i|$ ($i=1,2\cdots,q$) and $n_1\geq n_2\geq\cdots\geq n_q\geq1$. To find the feasible minimum distance signless Laplacian spectral radius, we first consider the graph $G'$, which is obtained from $G$ by joining $S$ and $G-S$ and by adding edges in $S$ and in  all components in $G-S$ so that $G_1,G_2,\cdots,G_q$ and $G[S]$ are cliques, i.e. $G'\cong K_s\vee(K_{n_1}\cup K_{n_2}\cup\cdots\cup K_{n_q})$. From Lemma \ref{largerrd}, it's easy to conclude that $\eta_1(G')\leq\eta_1(G)$.
	
	Next, let $n_1=n-s-q+1$ and $n_2=n_3=\cdots=n_q=1$, then we obtain a new graph $G''\cong K_s\vee K_{n-s+q-1}\vee \overline{K_{q-1}}$. Since $G'\cong G''$ if $n_1=1$, we only consider $n_1\geq3$. The quotient matrix of the distance signless Laplacian matrix $\mathcal{Q}(G'')$ according to the partition $\left\lbrace V(\overline{K_{q-1}}), V(K_{n-s-q+1}), V(K_s) \right\rbrace $ of $G''$, denoted by $\mathcal{B}$, equals
	\begin{equation*}
		\begin{bmatrix}
			2n-s+2q-6 & 2(n-s-q+1) & s \\
			2(q-1) & 2n-s-2 & s \\
			q-1 & n-s-q+1 & n+s-2
		\end{bmatrix}.
	\end{equation*}

	Let the vector \textbf{X} be an eigenvector of $\mathcal{Q}(G'')$ corresponding to $\eta_1(G'')$ and $x(v)$ be the entry of $\textbf{X}$ corresponding to vertex $v\in V(G'')$. By symmetry, all vertices in $V(\overline{K_{q-1}})$ (resp. $V(K_{n-s-q+1})$ or $V(K_s)$) have the same entries in $\textbf{X}$. Hence,
	we denote $x(v_1)=x_1$ for $v_1\in V(\overline{K_{q-1}})$, $x(v_2)=x_2$ for $v_2\in V(K_{n-s-q+1})$ and $x(v_3)=x_3$ for $v_3\in V(K_S)$. Note that the partition $\left\lbrace V(\overline{K_{q-1}}), V(K_{n-s-q+1}), V(K_s) \right\rbrace $ is equitable, we get $\eta_1(G'')(x_1,x_2,x_3)^t=B_1(x_1,x_2,x_3)^t$, i.e.,
\begin{equation*}
	\begin{cases}
		\eta_1(G'')x_1 = sx_3+2(n-s-q+1)x_2+(2n-s+2q-6)x_1,\\
		\eta_1(G'')x_2 = sx_3+(2n-s-2)x_2+2(q-1)x_1, \\
		\eta_1(G'')x_3 = (n+s-2)x_3+(n-s-q+1)x_2+(q-1)x_1.
	\end{cases}
\end{equation*}

By calculation, we have
\begin{equation*}
	x_1+x_2 = 2\left( \frac{\eta_1(G'')+4-n-q}{\eta_1(G'')+4+s-2n}\right)x_2.
\end{equation*}

Then by Lemma \ref{symm}, one can see
\begin{align*}
	&~~~~\eta_1(G')-\eta_1(G'')\\
	&\geq X^t(\mathcal{Q}(G')-\mathcal{Q}(G''))X\\
	&=2n_1\sum_{i=2}^{q}(n_i-1)x^2_2+\sum_{i=2}^{q}(n_i-1)\left[-x^2_1-2x_1x_2+(n-s-n_i-q+1)x^2_2+(n-s-n_i-q+2)x^2_2\right]\\
	&=\sum_{i=2}^{q}(n_i-1)\left[ 2(n-s+n_1-n_i-q+2)x^2_2-(x_1+x_2)^2\right]\\
	&=2\sum_{i=2}^{q}(n_i-1)\left[(n-s+n_1-n_i-q+2)-2\left( \frac{\eta_1(G'')+4-n-q}{\eta_1(G'')+4+s-2n}\right)^2\right]x_2^2.
\end{align*}
	
	Note that $n_1>n_2\geq n_2\geq n_3\cdots\geq n_q>1$ and $n_1-n_2\geq2$. To prove the inequality $\eta_1(G')-\eta_1(G'')>0$, we only prove $(n-s+n_1-n_2-q+2)-2\left( \frac{\eta_1(G'')+4-n-q}{\eta_1(G'')+4+s-2n}\right)^2>(n-s-q+4)-2\left( \frac{\eta_1(G'')+4-n-q}{\eta_1(G'')+4+s-2n}\right)^2>0$. From the construction of $G''$ and Lemma \ref{symm}, one can check $\eta_1(G'')>n+q-4$ and $\eta_1(G'')>2n-s-4$, then we have
	\begin{align}
		&~~~(n-s-q+4)-2\left( \frac{\eta_1(G'')+4-n-q}{\eta_1(G'')+4+s-2n}\right)^2>0\nonumber\\
		&\Leftrightarrow 1+\frac{n-s-q}{\eta_1(G'')+4+s-2n} <\sqrt{\frac{n-s-q+4}{2}}\nonumber\\
		&\Leftrightarrow \eta_1(G'')>2n-s-4+\frac{n-s-q}{\sqrt{\frac{n-s-q+4}{2}}-1}\label{inequ1}.
	\end{align}

	For convenience, we denote $r=\frac{n-s-q}{\sqrt{\frac{n-s-q+4}{2}}-1}$, then the inequality \eqref{inequ1} becomes $\eta_1(G'')>2n-s-4+r$. Now, we consider the determinant of matrix $(2n-s-4+r)I-\mathcal{B}$, where $I$ is an identity matrix.  If $\det\left((2n-s-4+r)I-\mathcal{B}\right)<0$, we can conclude $\eta_1(G'')=\rho(\mathcal{B})>2n-s-4+r$. 
	
	By calculation, we have
	\begin{equation*}
		\begin{split}
			&~~~\det\left((2n-s-4+r)I-\mathcal{B}\right)\\
			&=\left|\begin{matrix}
				n-2s-2+r & -(n-s-q+1) & -(q-1) \\
				-s & r-2 & -2(q-1) \\
				-s & -2(n-s-q+1) & -2(q-1)+r
			\end{matrix}\right|\\
			&=r\left|\begin{matrix}
				n-2s-2+r & -(n-s-q+1) \\
				-s & r-2
			\end{matrix}\right|+2(q-1)\left|\begin{matrix}
			n-\frac{3}{2}s-2+r & -(n-s-q)+\frac{1}{2}r] \\
			0 & -2(n-s-q)-r
		\end{matrix}\right|\\
		&=r\left[(n-2s-2)(r-2)+r^2-2r-(n-s-q+1)s\right]-2(q-1)(n-\frac{3}{2}s-2+r)(2n-2s-2q+r)\\
		&=\left[(n-2s-4)r^2-2(q-1)r^2-4(q-1)(n-\frac{3}{2}s-2)(n-s-q)\right]-2(n-2s-2)r-rs\\
		&~~~~-2(q-1)(n-\frac{3}{2}s-2)r+\left[r^2-(n-s-q)s-4(q-1)(n-s-q)\right]r.
		\end{split}
	\end{equation*}

	Obviously, we only need to check whether inequalities $(n-2s-4)r^2-2(q-1)r^2-4(q-1)(n-\frac{3}{2}s-2)(n-s-q)<0$ and $r^2-(n-s-q)s-4(q-1)(n-s-q)<0$ hold. Note that $(n-2s-4)r^2-2(q-1)r^2-4(q-1)(n-\frac{3}{2}s-2)(n-s-q)$ is negative when $n-2s-2q-2\leq0$. If $n-2s-2q-2>0$, based on the above discussion, we have
	\begin{align}
		&~~~~(n-2s-4)r^2-2(q-1)r^2-4(q-1)(n-\frac{3}{2}s-2)(n-s-q)<0\label{inequ2}\\
		&\Leftrightarrow(n-2s-2q-2)\frac{(n-s-q)^2}{\frac{n-s-q+6}{2}-2\sqrt{\frac{n-s-q+4}{2}}}-2(q-1)(n-\frac{3}{2}s-2)\frac{(n-s-q)^2}{\frac{n-s-q}{2}}<0\nonumber\\
		&\Leftrightarrow\frac{n-s-q}{n-s-q+6-2\sqrt{2}\sqrt{n-s-q+4}}<\frac{2(q-1)\left(n-2s-2q-2+\frac{s}{2}+2q\right)}{n-2s-2q-2}\nonumber\\
		&\Leftrightarrow\frac{1}{1+\frac{6}{n-s-q}-2\sqrt{2}\sqrt{\frac{1}{n-s-q}+\frac{4}{(n-s-q)^2}}}<2(q-1)+\frac{\frac{s}{2}+2q}{n-2s-2q-2}\nonumber.
	\end{align}
	
	Since $\frac{1}{1+\frac{6}{n-s-q}-2\sqrt{2}\sqrt{\frac{1}{n-s-q}+\frac{4}{(n-s-q)^2}}}$ attains its maximum if $n-s-q=4$, one can easily check that $\frac{1}{1+\frac{6}{4}-2\sqrt{2}\sqrt{\frac{1}{4}+\frac{4}{(4)^2}}}=2<2(q-1)+\frac{\frac{s}{2}+2q}{n-2s-2q-2}$, then the inequality \eqref{inequ2} holds.
	
	Analogously, it follows that
	\begin{align}
		&~~~~r^2-(n-s-q)s-4(q-1)(n-s-q)<0\label{inequ3}\\
		&\Leftrightarrow\frac{(n-s-q)^2}{\frac{n-s-q+6}{2}-2\sqrt{\frac{n-s-q+4}{2}}}-(2q+\frac{s}{2}-2)\frac{(n-s-q)^2}{\frac{n-s-q}{2}}<0\nonumber\\
		&\Leftrightarrow\frac{n-s-q}{n-s-q+6-2\sqrt{2}\sqrt{n-s-q+4}}<2q+\frac{s}{2}-2\nonumber\\
		&\Leftrightarrow\frac{1}{1+\frac{6}{n-s-q}-2\sqrt{2}\sqrt{\frac{1}{n-s-q}+\frac{4}{(n-s-q)^2}}}<2q+\frac{s}{2}-2\nonumber.
	\end{align}
	
	Since $\frac{1}{1+\frac{6}{n-s-q}-2\sqrt{2}\sqrt{\frac{1}{n-s-q}+\frac{4}{(n-s-q)^2}}}\leq2<2q+\frac{s}{2}-2$, we obtain the inequality \eqref{inequ3} holds. On the whole, we have prove that $\eta_1(G'')>2n-s-4+\frac{n-s-q}{\sqrt{\frac{n-s-q+4}{2}}-1}$, which implies that $(n-s+n_1-n_2-q+2)-2\left( \frac{\eta_1(G'')+4-n-q}{\eta_1(G'')+4+s-2n}\right)^2>(n-s-q+4)-2\left( \frac{\eta_1(G'')+4-n-q}{\eta_1(G'')+4+s-2n}\right)^2>0$. Therefore, for $n_1\geq3$, $\eta_1(G'')<\eta_1(G')$.

	Now, we suppose $q=s+2$ and get a new graph $G'''\cong K_s\vee K_{n-2s-1}\vee \overline{K_{s+1}}$. Then we consider the difference between $\eta_1(G'')$ and $\eta_1(G''')$. Since $G''\cong G'''$ if $q=s+2$, we let $q\geq s+4$. Similarly, we let the vector \textbf{Y} be an eigenvector of $\mathcal{Q}(G''')$ corresponding to $\eta_1(G''')$ and $y(v)$ be the entry of $\textbf{Y}$ corresponding to vertex $v\in V(G''')$. Suppose that $y(v_1)=y_1$ for $v_1\in V(\overline{K_{s+1}})$, $y(v_2)=y_2$ for $v_2\in V(K_{n-2s-1})$ and $y(v_3)=y_3$ for $v_3\in V(K_S)$. Combining Lemma \ref{symm} yields
	\begin{align*}
		&~~~~\eta_1(G'')-\eta_1(G''')\\
		&\geq Y^t(\mathcal{Q}(G'')-\mathcal{Q}(G'''))Y\\
		&=3(q-s-2)(n-s-q+1)y_2^2+(q-s-3)(q-s-2)y_2^2+(n-2s-2)(q-s-2)y_2^2\\
		&>0,
	\end{align*}
	which implies that $\eta_1(G''')<\eta_1(G'')$ if $q\geq s+4$.

	Let $G''''=K_1\vee K_{n-3}\vee\overline{K_{2}}$. In the remainder of this section, we will clarify that, in most cases, the graph $G''''$ has the minimum distance signless Laplacian spectral radius and contains no perfect matchings.
	 The quotient matrix of distance signless Laplacian matrix $\mathcal{Q}(G''')$ according to the partition $\left\lbrace V(K_s), V(K_{n-2s-1}), V(\overline{K_{s+1}}) \right\rbrace $ can be expressed as
	 \begin{equation*}
	 	\begin{bmatrix}
	 		n+s-2 & n-2s-1 & s+1 \\
	 		s & 2n-s-2 & 2(s+1) \\
	 		s & 2(n-2s-1) & 2n+s-2
	 	\end{bmatrix}.
	 \end{equation*}
 
 	The corresponding characteristic polynomial equals
 	\begin{equation*}
 		\begin{split}
 			f(x)&=x^3+(6-s-5n)x^2+(8n^2-ns-24n+8s^2+8s+16)x-4n^3+2n^2s+20n^2\\
 			&~~~~-8ns^2-14ns-32n-2s^3+14s^2+20s+16.
 		\end{split}
 	\end{equation*}
 
  Note that $G'''\cong G''''$ if $s=1$. Let $s=1$, the polynomial $f(x)$ becomes
  $\tilde{f}(x)=x^3+(5-5n)x^2+(8n^2-25n+32)x-4n^3+22n^2-54n+48$. The largest root of the equation $\tilde{f}(x)=0$ is denoted by $\theta(n)$. By calculation, we can get the explicit formula of $\theta(n)$, that is 
  
  \begin{equation*}
  	\small
  	\begin{split}
  		\theta(n)&=\frac{1}{3}\left( \frac{2}{ 141n^2-357n-2n^3-106+3\sqrt{3}\sqrt{-32n^5+543n^4-4450n^3+21095n^2-53208n+53440}}\right)^{\frac{1}{3}} n^2\\
  		&~~~+\left[ \frac{25}{3}\left( \frac{2}{ 141n^2-357n-2n^3-106+3\sqrt{3}\sqrt{-32n^5+543n^4-4450n^3+21095n^2-53208n+53440}}\right)^{\frac{1}{3}}+\frac{5}{3}\right] n\\
  		&~~~+\frac{1}{3}\left( \frac{141n^2-357n-2n^3-106+3\sqrt{3}\sqrt{-32n^5+543n^4-4450n^3+21095n^2-53208n+53440}}{2}\right)^{\frac{1}{3}}\\
  		&~~~-\frac{71}{3}\left( \frac{2}{ 141n^2-357n-2n^3-106+3\sqrt{3}\sqrt{-32n^5+543n^4-4450n^3+21095n^2-53208n+53440}}\right)^{\frac{1}{3}}-\frac{5}{3}.
  	\end{split}
  \end{equation*}
  
  By plugging the value $\theta(n)$ into $x$ of $f(x)$, we have
  \begin{align*}
  	f(\theta(n))&=\tilde{f}(\theta(n))+(1-s)x^2+(n-ns+8s^2+8s-16)x+2n^2s-2n^2-8ns^2\\
  	&~~~~-14ns+22n-2s^3+14s^2+20s-32\\
  	&=(1-s)\left[(\theta(n))^2+n\cdot\theta(n)-8(s+2)\cdot\theta(n)-2n^2+2n(4s+11)+2s(s-6)-32\right].
  \end{align*}
	
	Now, we suppose $s\geq2$ and $n\geq 2s+4$. By Lemma \ref{largerrd}, one can see $\theta(n)=\eta_1(G'''')>2n+s-2>2n>\frac{16+8s-n}{2}$. Hence
	\begin{align*}
		&~~~~(\theta(n))^2+(n-8s-16)\theta(n)-2n^2+2n(4s+11)+2s(s-6)-32\\
		&>4n^2+2n^2-16ns-32n-2n^2+8ns+22n+2s^2-12s-32\\
		&=4n^2-8ns-10n+2s^2-12s-32\\
		&\geq 4(2s+4)^2-8(2s+4)s-10(2s+4)+2s^2-12s-32\\
		&=2s^2-8\geq0,
	\end{align*}
	which implies that for $s\geq2$ and $n\geq2s+4$, $f(\theta(n))<0$, i.e. $\theta(n)=\eta_1(G'''')<\eta_1(G''')$.
	
	If $n=2s+2$, one can check that $G'''\cong K_s\vee K_1\vee\overline{K_{s+1}}$ contains a perfect matching, a contradiction. Let $G'''''\cong K_s\vee\overline{K_{s+2}}$ and the quotient matrix of distance signless Laplacian matrix $\mathcal{Q}(G''''')$ according to the  partition $\left\lbrace K_s, \overline{K_{s+2}} \right\rbrace $ equals
	\begin{equation*}
		\begin{bmatrix}
			n+s-2 & s+2\\
			s & 5s+4 
		\end{bmatrix}.
	\end{equation*}

	The corresponding characteristic polynomial equals
	\begin{equation*}
		g(x)=x^2-(n+6s+2)x+4n-8s+5ns+4s^2-8.
	\end{equation*}

	Easily, we obtain $2n+\sqrt{\frac{n(n+2)}{2}}-2$ is the largest root of the equation $g(x)=0$. 
	
	Finally, we calculate that $\theta(4)=6+2\sqrt{3}$,  $\theta(6)\approx15.4597>10+2\sqrt{6}$, $\theta(8)\approx20.8655>14+2\sqrt{10}$,  $\theta(10)\approx26.0148>18+2\sqrt{15}$, and $\theta(n)<2n+\sqrt{\frac{n(n+2)}{2}}-2$ for $n\geq12$.
	
	This completes the proof.
\end{proof}

\section{Proof of Theorem \ref{them2}}
\begin{proof}[Proof of Theorem \ref{them2}] Assume to the contrary that $G$ contains no perfect matchings. Let $G=(X,Y)$ be a connect balanced bipartite graph with $2n$ vertices. Note that $V(G)=X\cup Y$, $X\cap Y=\varnothing$, and $|X|=|Y|=n$. From Theorem \ref{Hall}, one can see there exists a subset $S\subseteq X$ such that $|N(S)|<|S|$. For convenience, we denote $s=|S|$ and $k=|N(S)|$, and then we have $1\leq k<s\leq n-1$. To find the connection between distance spectral radius and perfect matchings, Zhang and Lin \cite{zhang} constructed a new connected balanced bipartite graph, denoted by $\Gamma_{s,k}$ in here, which is obtained from $G$ by joining $S$ and $N(S)$, $X-S$ and $Y-N(S)$, and by adding all edges between $N(S)$ and $X-S$, i.e. $\Gamma_{s,k}\cong K_{n,n}-e(S,Y-N(S))$. By Lemma \ref{largerrd}, it is easy to see $\eta_1(\Gamma_{s,k})\leq\eta_1(G)$.
	
First we consider the difference between $\eta_1(\Gamma_{s,k})$ and $\eta_1(\Gamma_{s,s-1})$. Since $\Gamma_{s,k}\cong\Gamma_{s,s-1}$ if $k=s-1$, we let $1\leq k \leq s-2$. By the Perron-Frobenius theorem, we suppose that the postive vector $\textbf{Z}$ is an eigenvector of $\mathcal{Q}(\Gamma_{s,s-1})$ corresponding to $\eta_1(\Gamma_{s,s-1})$ and $z(v)$ is the entry of $\textbf{Z}$ corresponding to vertex $v\in V(\Gamma_{s,s-1})$. Analogously, we denote $z(v_1)=z_1$ for $v_1\in S$, $z(v_2)=z_2$ for $v_2\in X-S$, $z(v_3)=z_3$ for $v_3\in N(S)$ and $z(v_4)=z_4$ for $v_4\in Y-N(S)$. From Lemma \ref{symm}, we have
	\begin{align*}
		\eta_1(\Gamma_{s,k})-\eta_1(\Gamma_{s,s-1})&\geq Z^t(\mathcal{Q}(\Gamma_{s,k})-\mathcal{Q}(\Gamma_{s,s-1}))Z\\
		&=2s(s-k-1)z_1^2+4s(s-k-1)z_1z_3>0.
	\end{align*}
So, for $1\leq k \leq s-2$, $\eta_1(\Gamma_{s,s-1})<\eta_1(\Gamma_{s,k})$.

Next, we shall show that the balanced bipartite graph $\Gamma_{n-1,n-2}$ has the minimum distance signless Laplacian spectral radius and contains no perfect matchings. Note that $\Gamma_{s,s-1}\cong\Gamma_{n-1,n-2}$ when $s=n-1$. In the later, we only consider $2\leq s\leq n-2$. For the bipartite graph $\Gamma_{s,s-1}$, the quotient matrix of the distance signless Laplacian matrix $\mathcal{Q}(\Gamma_{s,s-1})$ according to the partition $\left\lbrace S, X-S, N(S), Y-N(S) \right\rbrace $ has the form
\begin{equation*}
	\begin{bmatrix}
		7n-2s-2 & 2n-2s & s-1 & 3n-3s+3 \\
		2s & 5n-2s-4 & s-1 & n-s+1 \\
		s & n-s & 3n+2s-6 & 2n-2s+3 \\
		3s & n-s & 2s-2 & 5n-2
	\end{bmatrix}.
\end{equation*}
The corresponding characteristic polynomial equals
\begin{align*}
	h(x)&=x^4+(2s-20n+14)x^3+(145n^2-38ns-214n+12s^2+10s+74)x^2\\
	&~~~~-(450n^3-190n^2s-1052n^2+82ns^2+189ns+775n-78s^2+2s-172)x\\
	&~~~~+504n^4-282n^3s-1656n^3+150n^2s^2+517n^2s+1951n^2-24ns^3-226ns^2\\
	&~~~~-220ns-938n+12s^4-24s^3+138s^2-46s+144.
\end{align*}

As for the bipartite graph $\Gamma_{n-1,n-2}$, the quotient matrix of the distance signless Laplacian matrix $\mathcal{Q}(\Gamma_{n-1,n-2})$ according to the partition $\left\lbrace S, X-S, N(S), Y-N(S) \right\rbrace $ equals
\begin{equation*}
	\begin{bmatrix}
		5n & 2 & n-2 & 6 \\
		2n-2 & 3n-2 & n-2 & 2 \\
		n-1 & 1 & 5n-8 & 5 \\
		3n-3 & 1 & 2n-4 & 5n-2
	\end{bmatrix}.
\end{equation*}

The corresponding characteristic polynomial equals
\begin{align*}
	\tilde{h}(x)&=x^4+(12-18n)x^3+(119n^2-190n+76)x^2-(342n^3-915n^2+826n-252)x\\
	&~~~~+360n^4-1383n^3+2026n^2-1362n+364.
\end{align*}

Let $\kappa(n)$ be the largest root of the equation $\tilde{h}(x)=0$. Plugging $\kappa(n)$ into $h(x)$ yields
\begin{align*}
	h(\kappa(n))&=\tilde{h}(\kappa(n))+(2s-2n-2)(\kappa(n))^3+(26n^2-38ns-24n+12s^2+10s-2)(\kappa(n))^2\\
	&~~~~-\left(108n^3-109n^2s+137n^2-82ns^2-189ns+51n+78s^2-2s-80\right)\kappa(n)+144n^4\\
	&~~~~-282n^3s-273n^3+150n^2s^2+517n^2s-75n^2-24ns^3-226ns^2-220ns+424n\\
	&~~~~+12s^4-24s^3+138s^2-46s-220\\
	&=(s-n+1)[2(\kappa(n))^3+(12s-26n-2)(\kappa(n))^2+(78s-29n-82ns+108n^2-80)\kappa(n)\\
	&~~~~-144n^3+138n^2s+129n^2-12ns^2-250ns+204n+12s^3-36s^2+174s-220 ]. 
\end{align*}

Denote $\varphi(n,s)=2(\kappa(n))^3+(12s-26n-2)(\kappa(n))^2+(78s-29n-82ns+108n^2-80)\kappa(n)-144n^3+138n^2s+129n^2-12ns^2-250ns+204n+12s^3-36s^2+174s-220$. It follows that
\begin{equation*}
	\frac{\partial \varphi(n,s)}{\partial s}=12(\kappa(n))^2+(78-82n)\kappa(n)+138n^2-24ns-250+36s^2-72s+174.
\end{equation*}

By Lemma \ref{symm}, we have
\begin{equation*}
	\kappa(n)=\eta_1(\Gamma_{n,n-1})=\max_{\textbf{x}\in \mathbb{R}^n}\frac{\textbf{x}^t\mathcal{Q}(\Gamma_{n,n-1})\textbf{x}}{\textbf{x}^{t}\textbf{x}}\geq\frac{\textbf{1}^t\mathcal{Q}(\Gamma_{n,n-1})\textbf{1}}{\textbf{1}^t\textbf{1}}=\frac{12n^2+8n-8}{2n}=6n+4-\frac{4}{n}>6n.
\end{equation*}

Since $\kappa(n)>6n>\frac{82n-78}{24}$, then we have
\begin{align*}
	&~~~~12(\kappa(n))^2+(78-82n)\kappa(n)+138n^2-24ns-250+36s^2-72s+174\\
	&>12\cdot(6n)^2+(78-82n)\cdot6n+138n^2-24ns-250+36s^2-72s+174\\
	&=78n^2+(218-24s)n+36s^2-72s+174
\end{align*}

Let $\psi(n,s)=78n^2+(218-24s)n+36s^2-72s+174$.  Since
\begin{equation*}
\frac{\partial \psi(n,s)}{\partial s} = 72s-72-24n = 0\Leftrightarrow s = \frac{n}{3}+1,
\end{equation*}
 one can see the function $\psi(n,s)$ attains its minimum if $s=\frac{n}{3}+1$, then we have $\psi(n,s)\geq\psi(n,\frac{n}{3}+1)=74n^2+194n+138>0$, which implies that $\frac{\partial \varphi(n,s)}{\partial s}>\psi(n,s)>0$. Hence, $\varphi(n,s)\geq \varphi(n,1)=2(\kappa(n))^3+(10-26n)(\kappa(n))^2+(108n^2-111n-2)\kappa(n)-144n^3+267n^2-58n-70$.

By program, we calculate that the function $\varphi(n,1)$ monotonically increasing for $n\geq3$, and 
\begin{align*}
	\varphi(n,1)&=2(\kappa(n))^3+(10-26n)(\kappa(n))^2+(108n^2-111n-2)\kappa(n)-144n^3+267n^2-58n-70\\
	&\geq2(\kappa(3))^3-68(\kappa(3))^2+637\kappa(3)-1729\approx414.17>0.
\end{align*}

Since $s-n+1<0$, we have $h(\kappa(n))=(s-n+1)\cdot \varphi(n,s)<(s-n+1)\cdot \varphi(n,1)<0$, implying $\eta_1(\Gamma_{n-1,n-2})<\eta_1(\Gamma_{s,s-1})$.

This completes the proof.

\end{proof}
\section*{Declaration of competing interest}
The authors declare that they have no known competing financial interests or personal relationships that could have appeared to influence the work reported in this paper.

\section*{Acknowledgments}\setlength{\baselineskip}{15pt}
 The authors would like to express their sincere gratitude to the referees for their careful reading and insightful suggestions.

\end{document}